\documentclass[12pt,reqno]{amsart}

\usepackage[english]{babel}
\usepackage[utf8]{inputenc}

\usepackage{faktor}
\usepackage{enumitem}
\usepackage{graphicx}
\usepackage{subfig}
\usepackage{tikz-cd}
\usetikzlibrary{arrows.meta}

\usepackage[T1]{fontenc} 
\usepackage{amsmath,amsthm,amsfonts,amssymb,dsfont,mathtools} 
\usepackage{epsfig,subfig}	
\usepackage{multirow,multicol}
\usepackage{tabularx} 
\usepackage{setspace} 
\usepackage{verbatim} 

\usepackage{color}
\usepackage[dvipsnames]{xcolor}

\usepackage[pdftex,%
bookmarks=true,
breaklinks=true,
bookmarksnumbered = true,
colorlinks= true,
urlcolor= green,
anchorcolor = yellow,
citecolor=blue,
]{hyperref}                   

\newtheoremstyle{theoremm}{}{}{\itshape}{}{\scshape}{.}{ }{}
\theoremstyle{theoremm}
\newtheorem{teo}{Theorem}[section]
\newtheorem{lema}[teo]{Lemma}
\newtheorem{prop}[teo]{Proposition}
\newtheorem{coro}[teo]{Corollary}

\newtheoremstyle{remark}{}{}{}{}{\scshape}{.}{ }{}
\theoremstyle{remark}
\newtheorem{defi}[teo]{Definition}
\newtheorem{obs}[teo]{Remark}

\newcommand{\aut}[1]{\ensuremath{\operatorname{\text{Aut}}\left({#1}\right)}}
\renewcommand{\ker}[1]{\ensuremath{\operatorname{\text{Ker}}\left({#1}\right)}}

\renewcommand{\to}{\ensuremath{\longrightarrow}}
\renewcommand{\mapsto}{\ensuremath{\longmapsto}}

\newcommand{\ang}[1]{\ensuremath{\left\langle #1\right\rangle}}
\newcommand{\set}[2]{\ensuremath{\left\{#1 \,\mid\, #2\right\}}}

\makeatletter
\def\@map#1#2[#3]{\mbox{$#1 \colon\thinspace #2 \to #3$}}
\def\map#1#2{\@ifnextchar [{\@map{#1}{#2}}{\@map{#1}{#2}[#2]}}
\makeatother

\numberwithin{equation}{section}
\allowdisplaybreaks

\usepackage{fancyhdr}
\begin{document}

\title{Permutational wreath pullbacks and framed braid-type groups}

\author[Enio Leite]{{\^E}nio Leite}
\address{Universidade Federal da Bahia, Departamento de Matem\'atica - IME, CEP:~40170-110 - Salvador, Brazil}
\email{leiteenio@gmail.com}

\author[Oscar Ocampo]{Oscar Ocampo}
\address{Universidade Federal da Bahia, Departamento de Matem\'atica - IME, CEP:~40170-110 - Salvador, Brazil}
\email{oscaro@ufba.br}



\date{\today}

\begin{abstract}

Let $\sigma\colon G \to S_n$ be a surjective homomorphism and let $H$ be a group. 
We introduce the \emph{permutational wreath pullback}
\[
H \wr_\sigma G = H^n \rtimes_\sigma G,
\]
where the action of $G$ on $H^n$ is induced by permutation of coordinates via $\sigma$, and undertake a systematic structural study of this construction. 
We determine the center and the abelianization in full generality. 
We further show that $H \wr_\sigma G$ admits a natural interpretation as the pullback of the classical wreath product $H \wr S_n$ along $\sigma$, providing a conceptual explanation for its functorial behavior. 
When $H$ is finitely generated abelian, we establish a criterion for the abelian kernel $H^n$ to be characteristic and for $H \wr_\sigma G$ to inherit the $R_\infty$-property from $G$; we verify this criterion for kernels arising from the virtual braid group $VB_n$ and the virtual twin group $VT_n$, obtaining new families of framed groups with the $R_\infty$-property. 
Rigidity results show that the abelian kernel, $n$, $H$, and $G$ are determined by the abstract group $H \wr_\sigma G$. 
Applications include uniform descriptions of classical, surface, virtual, and singular framed braid groups, and a reduction of splitting problems for framed surface braid groups to the classical Fadell--Neuwirth setting. 
\end{abstract}

\maketitle

\section{Introduction}\label{sec:intro}

Wreath products and semidirect products defined by permutation actions arise naturally in several areas of group theory and topology, particularly in geometric group theory. 
A classical example is the standard wreath product
\[
H \wr S_n = H^n \rtimes S_n,
\]
where the symmetric group $S_n$ acts on $H^n$ by permuting coordinates. Such constructions play a fundamental role in the study of groups acting on configurations, and appear naturally in contexts involving braid groups, mapping class groups, and configuration spaces.

In braid theory, this phenomenon appears prominently in the study of \emph{framed braid groups}. For the classical braid group $B_n$, the canonical surjective homomorphism
\[
\sigma\colon B_n \to S_n
\]
induces an action of $B_n$ on $\mathbb{Z}^n$ by permutation of coordinates, leading to the framed braid group
\[
\mathbb{Z}^n \rtimes B_n,
\]
introduced by Ko and Smolinsky~\cite{KS} and further studied in connection with $3$-manifolds and quantum invariants. 
This construction has been extended in several directions. Juyumaya and Lambropoulou~\cite{JL1,JL2} introduced and studied $p$-adic framed braid groups and their connections with knot invariants and Hecke-type algebras. Bellingeri and Gervais~\cite{BG} developed geometric models of framed braid groups on surfaces, showing that the structure of the framed group depends on the topology of the underlying surface. 

Framed versions of virtual, singular, and other braid-type groups were later developed in the doctoral thesis of the first named author~\cite{L}, where new families were introduced and a systematic study of their structural properties was carried out, encompassing both classical and surface braid groups. The present paper shows that all these constructions arise naturally from a single algebraic mechanism, namely permutational wreath pullbacks, thereby providing a unified and conceptual framework for framed braid-type groups.

Despite these developments, existing approaches are largely \emph{case-by-case} and often rely on geometric or presentation-based constructions. In particular, a unified structural framework encompassing these examples from a purely algebraic viewpoint has been missing. 
The purpose of this paper is to introduce such a framework and to show that it leads to a systematic and uniform treatment of framed braid-type groups, replacing case-by-case arguments by conceptual structural results.

Let $G$ be a group equipped with a surjective homomorphism
\[
\sigma\colon G \to S_n,
\]
and let $H$ be an arbitrary group. The action of $S_n$ on $H^n$ by permutation of coordinates induces, via $\sigma$, an action of $G$ on $H^n$. We define the associated \emph{permutational wreath pullback} as
\[
H \wr_\sigma G := H^n \rtimes_\sigma G.
\]
This construction simultaneously generalizes and unifies:
\begin{itemize}
\item classical framed braid groups ($H=\mathbb{Z}$, $G=B_n$),
\item framed braid groups on surfaces,
\item framed braid-type groups such as virtual and singular braid groups,
\item and, more generally, constructions with arbitrary coefficient groups $H$, including $p$-adic and related variants as in~\cite{JL1,JL2}.
\end{itemize}

The term “permutational wreath pullback” reflects the fact that this construction admits a canonical interpretation as a base change of the classical wreath product along the homomorphism $\sigma$. 
More precisely, in Theorem~\ref{teo:pullback} we establish the following fundamental description: there is a natural isomorphism
\[
H \wr_\sigma G \cong G \times_{S_n} (H \wr S_n),
\]
showing that $H \wr_\sigma G$ arises as the pullback of the classical wreath product along $\sigma$. 
This viewpoint explains both the naturality and the functoriality of the construction, and places framed braid-type groups within a unified categorical framework.

\noindent
\textbf{Main results.}
\begin{enumerate}
\item[\textbf{(A)}] \textbf{Structure.} 
We determine the center and the abelianization of $H \wr_\sigma G$ in full generality (Theorems~\ref{teo:center} and~\ref{teo:abelianization}):
\[
Z(H \wr_\sigma G) = \Delta(Z(H)) \times \bigl(Z(G)\cap \ker{\sigma}\bigr),
\qquad
(H \wr_\sigma G)^{ab} \cong H^{ab} \times G^{ab}.
\]
These results show that the structure of $H \wr_\sigma G$ is governed by a precise interaction between the internal algebraic properties of $H$ and the permutation action induced by $\sigma$.

\item[\textbf{(B)}] \textbf{Rigidity.} 
Under a natural structural condition $(\ast)$ (every abelian normal subgroup is contained in $H^n$), the subgroup $H^n$ is intrinsically determined as the largest abelian normal subgroup, and hence is characteristic in $H \wr_\sigma G$ (Theorem~\ref{teo:Hn_characteristic}). 
Moreover, the decomposition $H \wr_\sigma G = H^n \rtimes_\sigma G$ is rigid: both the abelian kernel and the quotient group are determined by the group structure (Theorem~\ref{teo:intrinsic_rigidity}). 
These results establish a strong form of structural rigidity.

\item[\textbf{(C)}] \textbf{Subgroups associated to the permutation kernel.}
Let $\pi\colon H \wr_\sigma G \to G$ be the canonical projection. 
In Proposition~\ref{prop:pure_splits} we show that the subgroup lying over $\ker{\sigma}$ splits as a direct product:
\[
\pi^{-1}(\ker{\sigma}) \cong H^n \times \ker{\sigma}.
\]
This subgroup is characteristic whenever $\ker{\sigma}$ is characteristic in $G$ (Theorem~\ref{teo:pure_characteristic}). 
We also derive consequences for twisted conjugacy and the $R_\infty$-property: under condition $(\ast)$, if $G$ has the $R_\infty$-property, then so does $H \wr_\sigma G$ (Theorem~\ref{teo:R_infty_general}).

\item[\textbf{(D)}] \textbf{Applications to braid-type groups.} 
We obtain explicit structural descriptions of classical, surface, virtual and singular framed braid groups, including their centers and abelianizations. 
We verify that the subgroups $\ker\sigma$ for the relevant epimorphisms $(\pi_K\colon VB_n\to S_n$, $\sigma\colon VT_n\to S_n$, $\theta\colon VT_n\to S_n)$ contain no non-trivial abelian normal subgroups of the corresponding ambient groups. 
Consequently, the associated framed groups satisfy condition $(\ast)$ and therefore inherit the $R_\infty$-property (Corollary~\ref{coro:R_infty_examples}).  
\end{enumerate}

Taken together, these results show that the permutational wreath pullback provides a unified algebraic framework that both generalizes and clarifies the structure of framed braid-type groups in a conceptual way. 
The paper is organized as follows. 
In Section~\ref{sec:wreath}, we introduce the construction and establish its basic properties. Section~\ref{sec:structure} contains the structural and rigidity results. In Section~\ref{sec:pullback}, we develop the pullback interpretation and functoriality. 
Finally, Section~\ref{sec:braid} applies the general theory to braid-type groups. 
We obtain explicit structural descriptions of classical and surface framed braid groups, 
and introduce framed analogues of virtual and singular braid groups. 
We also verify condition $(\ast\ast)$ (defined in Theorem~\ref{teo:conditionA_equiv}) for the kernels $KB_n$, $PVT_n$ and $KT_n$ arising from the virtual braid group $VB_n$ and the virtual twin group $VT_n$, leading to new families of framed groups with the $R_\infty$-property. Moreover, we show that splitting problems for framed surface braid groups reduce to the classical Fadell-Neuwirth setting, so that no new obstructions arise in the framed context. This provides a unified and conceptual framework encompassing several previously unrelated constructions.

\subsection*{Acknowledgments}

The second named author gratefully acknowledges the support of Eliane Santos, the staff of HCA, Bruno Noronha, Luciano Macedo, M\'arcio Isabella, Andreia de Oliveira Rocha, Andreia Gracielle Santana, Ednice de Souza Santos, and SMURB--UFBA (Servi\c{c}o M\'edico Universit\'ario Rubens Brasil Soares), whose support since July 2024 was essential in enabling the completion of this work. 
E.~L.~was partially supported by Funda\c c\~ao de Amparo \`a Pesquisa do Estado da Bahia (FAPESB). 
O.~O.~was partially supported by the National Council for Scientific and Technological Development (CNPq, Brazil) through a \textit{Bolsa de Produtividade} grant No.~305422/2022--7.


\section{Permutational wreath pullbacks}\label{sec:wreath}

Let $n\ge 2$ be a fixed integer. Throughout the paper, $S_n$ denotes the symmetric group acting on the set $\{1,\dots,n\}$ by permutation. 
All groups considered in this paper are assumed to be non-trivial unless explicitly stated otherwise.

\begin{defi}
Let $G$ and $H$ be groups and let $\sigma\colon G \to S_n$ be a surjective homomorphism. 
The action of $S_n$ on $H^n$ by permutation of coordinates induces, via $\sigma$, an action of $G$ on $H^n$ defined by
\[
g \cdot (h_1,\dots,h_n)
=
\bigl(h_{\sigma(g)(1)},\dots,h_{\sigma(g)(n)}\bigr),
\qquad g\in G.
\]
We define the \emph{permutational wreath pullback} of $H$ by $G$ (with respect to $\sigma$) to be the semidirect product
\[
H \wr_\sigma G := H^n \rtimes_\sigma G.
\]
\end{defi}

Informally, the group $H \wr_\sigma G$ may be viewed as the result of letting $G$ act on $n$ copies of $H$ by permuting the coordinates according to the permutation representation induced by $\sigma$. 
From this perspective, the construction interpolates between semidirect products and wreath products: it behaves like a wreath product, but with the symmetric group replaced by an arbitrary group acting via its permutation representation.

Equivalently, $H \wr_\sigma G$ is obtained by pulling back the canonical permutational extension
\[
1 \to H^n \to H \wr S_n \to S_n \to 1
\]
along $\sigma\colon G \to S_n$. In this sense, it provides a canonical way to transfer the permutational structure of $S_n$ to any group equipped with a permutation representation.

Elements of $H \wr_\sigma G$ will be written as pairs $(\mathbf{h},g)$, where $\mathbf{h}=(h_1,\dots,h_n)\in H^n$ and $g\in G$. 
The multiplication is given by
\begin{equation}\label{eq:multiplication}
(\mathbf{h},g)(\mathbf{k},x)
=
\bigl(\mathbf{h}\cdot g\mathbf{k},\, gx\bigr),
\end{equation}
where $g\mathbf{k}$ denotes the permutational action of $g$ on $\mathbf{k}$.

\noindent
The group $H \wr_\sigma G$ fits into a natural short exact sequence
\begin{equation}\label{eq:extension}
1 \longrightarrow H^n 
\longrightarrow H \wr_\sigma G 
\longrightarrow G 
\longrightarrow 1,
\end{equation}
where the projection is given by $(\mathbf{h},g)\mapsto g$. 
The subgroup $H^n$ is normal in $H \wr_\sigma G$, and the subgroup
\[
\{(1,g)\mid g\in G\}
\]
is naturally isomorphic to $G$. Thus $H \wr_\sigma G$ contains canonical copies of both $H^n$ and $G$.

\begin{prop}\label{prop:dependence_sigma}
Let $(G_1,\sigma_1)$ and $(G_2,\sigma_2)$ be groups equipped with surjective homomorphisms
to $S_n$. If there exists an isomorphism $f\colon G_1 \to G_2$ such that
\[
\sigma_2 \circ f = \sigma_1,
\]
then there is a natural isomorphism
\[
H \wr_{\sigma_1} G_1 \cong H \wr_{\sigma_2} G_2.
\]
\end{prop}

\begin{proof}
Define
\[
\Phi\colon H^n \rtimes_{\sigma_1} G_1 \longrightarrow H^n \rtimes_{\sigma_2} G_2,
\qquad
(\mathbf{h},g) \longmapsto (\mathbf{h},f(g)).
\]
Since $\sigma_2(f(g))=\sigma_1(g)$, the actions of $g$ and $f(g)$ on $H^n$ coincide,
and hence $\Phi$ preserves multiplication. It is clearly bijective, with inverse
$(\mathbf{h},x)\mapsto (\mathbf{h},f^{-1}(x))$.
\end{proof}

In other words, the isomorphism type of $H \wr_\sigma G$ depends only on the permutation representation of $G$ up to equivalence.

\begin{obs}
Proposition~\ref{prop:dependence_sigma} shows that the construction depends on the pair
$(G,\sigma)$ rather than on $G$ alone.

In general, two different surjective homomorphisms $\sigma_1,\sigma_2\colon G \to S_n$
may yield non-isomorphic permutational wreath pullbacks, even when defined on the same underlying group $G$. 
This reflects the fact that the isomorphism type is determined by the induced action of $G$
on $H^n$.
\end{obs}

We conclude this section with a basic finiteness property.

\begin{prop}\label{prop:finitely_presented}
If $G$ and $H$ are finitely generated (respectively finitely presented), then $H \wr_\sigma G$ is finitely generated (respectively finitely presented).
\end{prop}

\begin{proof}
Since $H^n$ is finitely generated (respectively finitely presented) whenever $H$ is, and $H \wr_\sigma G$ is a semidirect product of $H^n$ by $G$, the result follows from standard results on semidirect products (see, for instance, \cite{MKS}).
\end{proof}


\section{Structural properties}\label{sec:structure}

Throughout this section, we assume that $\sigma\colon G \to S_n$ is surjective. Certain results will require additional restrictions on $n$, which will be specified when needed. 
In this section we study structural aspects of the permutational wreath pullback
\[
W := H \wr_\sigma G = H^n \rtimes_\sigma G,
\]
where $\sigma\colon G \to S_n$ is assumed to be surjective and $n\ge 2$.

\begin{obs}\label{rem:delta}
We shall frequently consider the diagonal embedding
\[
\Delta \colon H \to H^n, 
\qquad 
\Delta(h)=(h,\dots,h),
\]
and denote by $\Delta(H)$ its image. The subgroup $\Delta(H)$ is invariant under the action of $G$, since permutation of coordinates preserves diagonal elements.
\end{obs}

\begin{lema}\label{lem:fixed_points}
Suppose that $\sigma$ is surjective. Then the fixed-point subgroup
\[
(H^n)^{G}
=
\set{\mathbf{h}\in H^n}{g\mathbf{h}=\mathbf{h}\ \text{for all } g\in G}
\]
coincides with $\Delta(H)$.
\end{lema}

\begin{proof}
Since $\sigma(G)=S_n$, the action of $G$ on $H^n$ contains all permutations of coordinates. 
Thus $\mathbf{h}=(h_1,\dots,h_n)$ is fixed by $G$ if and only if it is invariant under all permutations in $S_n$, which happens precisely when $h_1=\dots=h_n$. 
Hence $(H^n)^G=\Delta(H)$.
\end{proof}

\subsection{The center and abelianization}

We begin with a description of the center of $W$ in full generality. 
To do this, we consider the subgroup $\Delta(H)$ described in Remark~\ref{rem:delta}. 

\begin{teo}\label{teo:center}
Let $W = H \wr_\sigma G$ and suppose that $\sigma\colon G \to S_n$ is surjective. Then
\[
Z(W) = \Delta(Z(H)) \times \bigl(Z(G)\cap \ker{\sigma}\bigr).
\]
In particular, if $Z(G)\subseteq \ker{\sigma}$, then
\[
Z(W)=\Delta(Z(H))\times Z(G).
\]
\end{teo}

\begin{proof}
Write elements of $W$ as $(\mathbf{h},g)$ with $\mathbf{h}\in H^n$ and $g\in G$.

Suppose $(\mathbf{h},g)\in Z(W)$. First, commuting with $(1,x)$ for arbitrary $x\in G$ gives
\[
(\mathbf{h},g)(1,x)=(\mathbf{h},gx)
\quad\text{and}\quad
(1,x)(\mathbf{h},g)=(x\mathbf{h},xg).
\]
Hence $g\in Z(G)$ and $\mathbf{h}=x\mathbf{h}$ for all $x\in G$.

Since $\sigma$ is surjective, the action of $G$ on $H^n$ contains all permutations of coordinates. 
Thus $\mathbf{h}$ is invariant under the full symmetric group $S_n$, and by Lemma~\ref{lem:fixed_points} we have $\mathbf{h}\in \Delta(H)$.

Next, commuting with arbitrary elements $(\mathbf{k},1)\in H^n$ yields
\[
(\mathbf{h},g)(\mathbf{k},1)
=
(\mathbf{h}\cdot g\mathbf{k},g)
=
(\mathbf{k}\cdot \mathbf{h},g)
=
(\mathbf{k},1)(\mathbf{h},g),
\]
and hence
\[
\mathbf{h}\cdot g\mathbf{k} = \mathbf{k}\cdot \mathbf{h}
\quad \text{for all } \mathbf{k}\in H^n.
\]
Rewriting, we obtain
\[
g\mathbf{k} = \mathbf{h}^{-1}\mathbf{k}\mathbf{h}.
\]

Write $\mathbf{h}=(a,\dots,a)\in \Delta(H)$. Then conjugation by $\mathbf{h}$ acts coordinatewise:
\[
\mathbf{h}^{-1}\mathbf{k}\mathbf{h}
=
(a^{-1}k_1 a,\dots,a^{-1}k_n a),
\]
whereas $g\mathbf{k}$ is obtained by permuting the coordinates via $\sigma(g)$:
\[
g\mathbf{k}
=
(k_{\sigma(g)(1)},\dots,k_{\sigma(g)(n)}).
\]

If $\sigma(g)\neq 1$, then there exist indices $i\neq j$ with $\sigma(g)(i)=j$.
Since $H$ is non-trivial, choose $x\in H$ with $x\neq 1$, and define $\mathbf{k}\in H^n$ by
\[
k_j=x,\quad \text{and } k_\ell=1 \text{ for } \ell\neq j.
\]
Then the $i$-th coordinate of $g\mathbf{k}$ is $x$, while the $i$-th coordinate of
$\mathbf{h}^{-1}\mathbf{k}\mathbf{h}$ is $a^{-1}1a=1$, a contradiction.
Hence $\sigma(g)=1$, and thus $g\in \ker\sigma$.

It follows that $g$ acts trivially on $H^n$, and the relation above reduces to
\[
\mathbf{h}\mathbf{k}=\mathbf{k}\mathbf{h}
\quad \text{for all } \mathbf{k}\in H^n,
\]
so $\mathbf{h}\in Z(H^n)$.
Since $\mathbf{h}\in \Delta(H)$, we conclude that $\mathbf{h}\in \Delta(Z(H))$.

Conversely, let $\mathbf{h}\in \Delta(Z(H))$ and $g\in Z(G)\cap\ker{\sigma}$.
Since $g\in \ker\sigma$, it acts trivially on $H^n$, and hence
\[
(\mathbf{h},g)(\mathbf{k},1)=(\mathbf{h}\mathbf{k},g)
\quad\text{and}\quad
(\mathbf{k},1)(\mathbf{h},g)=(\mathbf{k}\mathbf{h},g)
\]
for all $\mathbf{k}\in H^n$.
As $\mathbf{h}\in Z(H^n)$, these are equal, so $(\mathbf{h},g)$ commutes with $H^n$.

Moreover, since $g\in Z(G)$ and $\mathbf{h}$ is fixed under the action of $G$,
we also have
\[
(\mathbf{h},g)(1,x)=(\mathbf{h},gx)=(\mathbf{h},xg)=(1,x)(\mathbf{h},g)
\]
for all $x\in G$.
Thus $(\mathbf{h},g)$ commutes with both $H^n$ and $G$, and hence lies in $Z(W)$.

The final statement follows immediately from the identity
\[
Z(G)\cap\ker\sigma = Z(G)
\]
under the assumption $Z(G)\subseteq \ker\sigma$.
\end{proof}

\begin{obs}
Theorem~\ref{teo:center} shows that
\[
Z(W)\cap H^n = \Delta(Z(H)).
\]
In particular, the central contribution arising from the subgroup $H^n$
is entirely determined by the diagonal copy of $Z(H)$.
Thus even when $H$ is non-abelian, only its center contributes to the center of $W$.
\end{obs}

We now determine the abelianization of $W$.

\begin{teo}\label{teo:abelianization}
Let $W = H \wr_\sigma G$ with $\sigma$ surjective. Then there is a natural isomorphism
\[
W^{ab} \cong H^{ab} \times G^{ab}.
\]
\end{teo}

\begin{proof}
Let $N:=H^n\lhd W$. Since $W=N\rtimes G$, it follows from the standard description of the abelianization of a semidirect product (see, e.g., \cite[Chapter~I.5]{MKS}) that
\[
W^{ab}\cong \frac{N}{[N,W]} \times G^{ab},
\]
where $[N,W]$ is the subgroup of $N$ generated by $[N,N]$ and all elements of the form $(g\cdot \mathbf{n})\mathbf{n}^{-1}$ with $g\in G$ and $\mathbf{n}\in N$.

Now
\[
N^{ab}\cong (H^{ab})^n.
\]
The action of $G$ on $N^{ab}$ factors through $\sigma(G)=S_n$ and permutes the coordinates. 
Hence the subgroup generated by all elements of the form
\[
(g\cdot \mathbf{a})\mathbf{a}^{-1},
\qquad g\in G,\ \mathbf{a}\in (H^{ab})^n,
\]
identifies all coordinates in $(H^{ab})^n$, so that all components become equal.
Therefore 
\[
\frac{(H^{ab})^n}{\langle (g\cdot \mathbf{a})\mathbf{a}^{-1} \mid g\in G,\ \mathbf{a}\in (H^{ab})^n\rangle}
\cong H^{ab}.
\]

Since
\[
\frac{N}{[N,W]}
\cong
\frac{N^{ab}}{\langle (g\cdot \mathbf{a})\mathbf{a}^{-1} \mid g\in G,\ \mathbf{a}\in N^{ab}\rangle},
\]
it follows that
\[
\frac{N}{[N,W]}\cong H^{ab}.
\]
Consequently,
\[
W^{ab}\cong H^{ab}\times G^{ab}.
\]
\end{proof}

Theorem~\ref{teo:abelianization} shows that the contribution of $H$ to the abelianization of $W$ depends only on $H^{ab}$. In particular, if $H$ is perfect, then 
\[
W^{ab}\cong G^{ab}.
\]

\subsection{On condition {$(\ast)$}}\label{subsec:conditionA}

In subsequent sections, several rigidity and characteristicity results rely on the following structural hypothesis.

\noindent
\textbf{($\ast$)} Every abelian normal subgroup of $W=H \wr_\sigma G$ is contained in $H^n$.

Condition $(\ast)$ expresses that $H^n$ contains every abelian normal subgroup of $W$,
and hence is the maximal abelian normal subgroup of $W$. 
In this subsection we show that, for $n\ge 5$, condition {$(\ast)$} is equivalent to a natural condition on the base group $G$.

Throughout, let
\[
W = H \wr_\sigma G = H^n \rtimes_\sigma G,
\]
where $\sigma\colon G \to S_n$ is surjective and $n\ge 5$.

\begin{lema}\label{lem:abel_pr_kernel}
Let $A \lhd W$ be an abelian normal subgroup. Then
\[
\pi(A) \le \ker{\sigma},
\]
where $\pi\colon W \to G$ is the canonical projection.
\end{lema}

\begin{proof}
Since $A$ is abelian and normal in $W$, its image $\pi(A)$ is an abelian normal subgroup of $G$. Applying $\sigma$, we see that $\sigma(\pi(A))$ is an abelian normal subgroup of $S_n$. 
Since $n\ge 5$, the symmetric group $S_n$ has no non-trivial abelian normal subgroup (its only non-trivial normal subgroup is $A_n$, which is non-abelian). Therefore
\[
\sigma(\pi(A))=1,
\]
and hence $\pi(A)\le \ker{\sigma}$.
\end{proof}

Lemma~\ref{lem:abel_pr_kernel} reduces the study of abelian normal subgroups of $W$ to the subgroup $\ker{\sigma}\le G$. 
We now show that, for $n\ge 5$, condition $(\ast)$ is equivalent to a purely group-theoretic condition on $G$.

\begin{teo}\label{teo:conditionA_equiv}
Let $n\ge 5$. Then condition {$(\ast)$} is equivalent to the following condition:

\noindent
\textbf{($\ast \ast$)} The subgroup $\ker{\sigma}$ contains no non-trivial abelian normal subgroup of $G$. 
\end{teo}

\begin{proof}
Assume first that {$(\ast)$} holds.
Let $B \le \ker{\sigma}$ be an abelian normal subgroup of $G$.
Then
\[
\widetilde{B}:=\{(1,b)\mid b\in B\}\le W
\]
is an abelian normal subgroup of $W$. By {$(\ast)$}, we deduce $\widetilde{B}\le H^n$, which forces $B=1$, since $H^n\cap \{(1,b)\mid b\in G\}=\{1\}$. 
Thus {$(\ast \ast)$} holds.

Conversely, assume {$(\ast \ast)$} and let $A \lhd W$ be abelian. From Lemma~\ref{lem:abel_pr_kernel}, we have $\pi(A)\le \ker{\sigma}$. Since $\pi(A)$ is abelian and normal in $G$, condition {$(\ast \ast)$} implies $\pi(A)=1$.
Therefore $A\le \ker{\pi}=H^n$. Hence {$(\ast)$} holds.
\end{proof}

Theorem~\ref{teo:conditionA_equiv} provides a practical criterion that can be verified independently in specific families of groups. 
In braid-type settings, the homomorphism $\sigma\colon G \to S_n$ is typically not unique,
and different choices may lead to distinct subgroups $\ker{\sigma}$.
Thus condition~$(\ast)$ reduces to determining whether the chosen subgroup
$\ker{\sigma}$ contains a non-trivial abelian normal subgroup of $G$.

This criterion does not apply to the classical braid group $B_n$, since
\[
Z(B_n)=Z(P_n)=\langle \Delta_n^2\rangle \le P_n,
\]
where $\Delta_n^2$ denotes the full twist. Thus condition $(\ast\ast)$ fails for $B_n$.

More generally, Theorem~\ref{teo:conditionA_equiv} shows that the validity of $(\ast)$
is governed entirely by the normal subgroup structure of $\ker{\sigma}$ in $G$.
In particular, it reduces the problem to a purely group-theoretic condition on
$\ker{\sigma}$, namely the absence of non-trivial abelian normal subgroups. 
This reflects the fact that the permutational wreath pullback depends essentially on the choice of the permutation representation $\sigma$, as discussed in Section~\ref{sec:wreath}.

\subsection{Characteristicity of the permutational kernel and the $R_\infty$-property}

Throughout this subsection, let $H$ be a finitely generated abelian group, let $\sigma\colon G \to S_n$ be a surjective homomorphism, and assume that $n\ge 3$.  
We identify conditions ensuring that the normal subgroup
\[
H^n \lhd W=H \wr_\sigma G
\]
is characteristic, and derive consequences for the $R_\infty$-property. 
To formalize this intrinsic viewpoint, we introduce the following subgroup.

\begin{defi}\label{def:ax}
For a group $X$, define
\[
\mathcal{A}(X)
=
\langle A \le X \mid A \text{ is abelian and normal in } X \rangle.
\]
\end{defi}

We shall work under the following structural hypothesis:

\noindent
\textbf{($\ast$)} Every abelian normal subgroup of $W$ is contained in $H^n$.

For $n\ge 5$, condition {$(\ast)$} may be verified using Theorem~\ref{teo:conditionA_equiv}. 
This condition expresses that $H^n$ contains every abelian normal subgroup of $W$, and hence is intrinsically determined by the group.

\begin{prop}\label{prop:A_of_W}
Assume {$(\ast)$}. Then 
\[
\mathcal{A}(W)=H^n.
\]
In particular, $H^n$ is the largest abelian normal subgroup of $W$.
\end{prop}

\begin{proof}
Since $H^n$ is abelian and normal in $W$, we deduce $H^n \le \mathcal{A}(W)$.
Conversely, by assumption every abelian normal subgroup of $W$ is contained in $H^n$, so the subgroup generated by all of them is also contained in $H^n$. Hence $\mathcal{A}(W)=H^n$.
\end{proof}

Recall that a subgroup $K \le G$ is \emph{characteristic} if $\varphi(K)=K$ for every automorphism $\varphi\in\aut{G}$.

\begin{teo}\label{teo:Hn_characteristic}
Assume {$(\ast)$}. Then $H^n$ is characteristic in $W$.
\end{teo}

\begin{proof}
From Proposition~\ref{prop:A_of_W}, we get $H^n=\mathcal{A}(W)$. 
Since $\mathcal{A}(W)$ is defined purely in terms of the group structure of $W$, it is invariant under every automorphism of $W$.
\end{proof}

\begin{obs}\label{obs:intrinsic_vs_sigma}
Theorem~\ref{teo:Hn_characteristic} highlights a fundamental distinction between intrinsic and representation-dependent features of the construction.

The subgroup $H^n$ is intrinsically characterized as $\mathcal{A}(W)$, and hence is invariant under all automorphisms of $W$, independently of any chosen decomposition.

In contrast, the validity of $(\ast)$ depends on the permutation representation $\sigma\colon G \to S_n$, through the induced action of $G$ on $H^n$. Consequently, different choices of $\sigma$ for a fixed group $G$ may lead to non-isomorphic groups $H \wr_\sigma G$ with distinct normal subgroup structures.

This interplay between intrinsic and representation-dependent properties is a recurring theme throughout the paper.
\end{obs}

We now recall some definitions and a general fact on twisted conjugacy classes. 
Consider a group $G$ and an endomorphism $\alpha$ of $G$. We say that two elements $x$ and $y$ of $G$ are twisted conjugate (via $\alpha$) if and only if there exists a $z\in G$ such that $x = z y \alpha(z)^{-1}$. It is easy to see that the relation of being twisted conjugate is an equivalence relation and the number of equivalence classes (also referred to as Reidemeister classes) is called the Reidemeister number $R(\alpha)$ of $\alpha$. This Reidemeister number is either a positive integer or $\infty$. 
A group $G$ has the $R_\infty$-property if every automorphism $\varphi\in\aut{G}$ has infinitely many Reidemeister conjugacy classes.

Reidemeister numbers originate in Nielsen--Reidemeister fixed point theory, where one studies fixed point classes of selfmaps.
For a continuous map $f\colon X \to X$, one has
\[
R(f)=R(f_\ast),
\]
where $f_\ast\colon \pi_1(X)\to \pi_1(X)$ is the induced endomorphism.
This connection motivates the study of groups with the $R_\infty$-property; see for instance \cite{DGO1,DGO2,DGO3,FG}.

\begin{lema}[{\cite[Lemma~26]{DGO2}}]\label{lem:R_infty_extension}
Let
\[
1 \to N \to E \to Q \to 1
\]
be a short exact sequence with $N$ characteristic in $E$. If $Q$ has the $R_\infty$-property, then $E$ has the $R_\infty$-property.
\end{lema}

Combining Theorem~\ref{teo:Hn_characteristic} with Lemma~\ref{lem:R_infty_extension}, we obtain:

\begin{teo}\label{teo:R_infty_general}
Assume {$(\ast)$}. If $G$ has the $R_\infty$-property, then
\[
W = H \wr_\sigma G
\]
has the $R_\infty$-property.
\end{teo}

\begin{proof}
From Theorem~\ref{teo:Hn_characteristic}, the subgroup $H^n$ is characteristic in $W$.
Since $W/H^n \cong G$, the conclusion follows from Lemma~\ref{lem:R_infty_extension}.
\end{proof}

\begin{obs}\label{rem:braid_type_examples}

The general criterion above applies to several braid-type groups. 
In particular, for suitable epimorphisms onto $S_n$, condition $(\ast\ast)$ can be verified for kernels arising from the virtual braid group and the virtual twin group, leading to new families of permutational wreath pullbacks with the $R_\infty$-property.

These verifications require additional structural arguments and are developed in detail in Subsection~\ref{subsec:ast_ast}.

Once condition $(\ast)$ is verified for a given permutation representation $\sigma$, the corresponding framed group inherits the $R_\infty$-property from $G$ (see Theorem~\ref{teo:R_infty_general}). By Theorem~\ref{teo:conditionA_equiv}, verifying $(\ast)$ reduces to checking that $\ker{\sigma}$ contains no non-trivial abelian normal subgroup of $G$. Thus, the $R_\infty$-property for the framed group is governed entirely by the normal subgroup structure of the permutation representation $\sigma$.
\end{obs}

\subsection{Characteristic subgroups, splitting, and intrinsic rigidity}

In this subsection we study structural and intrinsic properties of the subgroup
of $W=H \wr_\sigma G$ lying over $\ker{\sigma}$, together with consequences
for automorphisms and rigidity phenomena.

Let
\[
P_\sigma=\ker{\sigma}\le G,
\qquad
PW=\pi^{-1}(P_\sigma)\le W,
\]
where $\pi\colon W\to G$ is the canonical projection.

We begin by describing the subgroup lying over $\ker{\sigma}$ and show that it admits a natural splitting as a direct product.

\begin{prop}\label{prop:pure_splits}
The subgroup $PW$ splits as a direct product:
\[
PW \cong H^n \times P_\sigma.
\]
\end{prop}

\begin{proof}
Since $P_\sigma=\ker{\sigma}$ acts trivially on $H^n$, the semidirect product
$H^n \rtimes P_\sigma$ reduces to the direct product $H^n \times P_\sigma$.
\end{proof}

For the standard epimorphism $B_n\twoheadrightarrow S_n$, one has $P_\sigma=P_n$, and Proposition~\ref{prop:pure_splits} recovers the classical splitting of the pure framed braid group as
\[
\mathbb{Z}^n \times P_n.
\]
The same argument applies to any braid-type group equipped with a natural epimorphism to $S_n$.

We next investigate conditions under which this subgroup is preserved by automorphisms, that is, when it is characteristic in $W$.

\begin{prop}\label{prop:induced_on_quotient}
Assume $(\ast)$. Then every automorphism $\Phi\in\mathrm{Aut}(W)$ induces an automorphism
\[
\bar\Phi\in \mathrm{Aut}(W/H^n)\cong \mathrm{Aut}(G).
\]
\end{prop}

\begin{proof}
Since $H^n=\mathcal{A}(W)$ is characteristic in $W$, it is preserved by every automorphism,
and hence automorphisms of $W$ descend to automorphisms of the quotient $W/H^n\cong G$.
\end{proof}

Proposition~\ref{prop:induced_on_quotient} shows that automorphisms of $W$ induce automorphisms of $G$.
This observation allows us to lift characteristicity from $G$ to the subgroup $PW$.

\begin{teo}\label{teo:pure_characteristic}
Assume $(\ast)$ and suppose that $P_\sigma=\ker{\sigma}$ is characteristic in $G$.
Then the subgroup
\[
PW=\pi^{-1}(P_\sigma)\cong H^n\times P_\sigma
\]
is characteristic in $W$.
\end{teo}

\begin{proof}
Let $\Phi\in\mathrm{Aut}(W)$.
By Proposition~\ref{prop:induced_on_quotient}, $\Phi$ induces an automorphism
\[
\bar\Phi\in\mathrm{Aut}(G).
\]
Since $P_\sigma$ is characteristic in $G$, we have $\bar\Phi(P_\sigma)=P_\sigma$.
Taking preimages under $\pi$ yields $\Phi(PW)=PW$.
\end{proof}

We now turn to intrinsic rigidity properties of the permutational wreath pullback, showing that under condition $(\ast)$ the abelian kernel and the quotient group are determined by the group structure of $W$. 
Let $H$ be a finitely generated abelian group and write
\[
H \cong \mathbb{Z}^r \oplus T,
\]
where $r=\mathrm{rk}(H)$ and $T$ is finite.
Assume that $\sigma\colon G \to S_n$ is surjective and that $n\ge 3$.

Recall that for a group $X$, $\mathcal{A}(X)$ denotes the subgroup generated by all abelian normal subgroups of $X$, see Definition~\ref{def:ax}. 
As in the previous subsection, we consider condition~\textbf{($\ast$)}: Every abelian normal subgroup of $W$ is contained in $H^n$. 
Under this hypothesis, the subgroup $H^n$ is intrinsically determined by $W$ via $\mathcal{A}(W)$.
For $n,m \ge 5$, condition {$(\ast)$} may be verified using Theorem~\ref{teo:conditionA_equiv}.

\begin{prop}\label{prop:A_invariant}
If $\Phi\colon X \to Y$ is a group isomorphism, then
\[
\Phi(\mathcal{A}(X))=\mathcal{A}(Y).
\]
\end{prop}

\begin{proof}
If $A\lhd X$ is abelian, then $\Phi(A)\lhd Y$ is abelian, and the result follows.
\end{proof}

The following result shows that, under condition {\rm($\ast$)}, the permutational wreath pullback decomposition is intrinsically determined.

\begin{teo}\label{teo:intrinsic_rigidity}
Let
\[
W = H^n \rtimes_\sigma G,
\qquad
W' = K^m \rtimes_{\sigma'} G',
\]
where $H,K$ are finitely generated abelian groups and $n,m\ge 3$.
Assume that both $W$ and $W'$ satisfy $(\ast)$.

If $\Phi\colon W \to W'$ is a group isomorphism, then:
\begin{multicols}{3}
    
\begin{enumerate}
\item $\Phi(H^n)=K^m$;
\item $H^n \cong K^m$;
\item $\mathrm{rk}(H)\, n = \mathrm{rk}(K)\, m$;
\item $T(H)^n \cong T(K)^m$;
\item $G \cong G'$.
\end{enumerate}
\end{multicols}
\end{teo}

\begin{proof}
From Proposition~\ref{prop:A_of_W} applied to $W$ and $W'$, we obtain
\[
\mathcal{A}(W)=H^n,
\qquad
\mathcal{A}(W')=K^m.
\]
Since $\Phi$ is an isomorphism, Proposition~\ref{prop:A_invariant} implies
\[
\Phi(H^n)=\Phi(\mathcal{A}(W))=\mathcal{A}(W')=K^m.
\]
Hence $H^n\cong K^m$.

Writing
\[
H^n \cong \mathbb{Z}^{rn}\oplus T(H)^n,
\qquad
K^m \cong \mathbb{Z}^{sm}\oplus T(K)^m,
\]
where $r=\mathrm{rk}(H)$ and $s=\mathrm{rk}(K)$, we obtain
\[
rn=sm
\]
and
\[
T(H)^n \cong T(K)^m.
\]

Finally, since
\[
W/H^n \cong G,
\qquad
W'/K^m \cong G',
\]
and $\Phi(H^n)=K^m$, the isomorphism $\Phi$ induces an isomorphism
\[
G \cong G'.
\]
\end{proof}

As a direct consequence, in the special case where $H=\mathbb{Z}$, the above rigidity result simplifies considerably and yields the following.

\begin{coro}\label{coro:intrinsic_rigidity_Z}
Let
\[
W=\mathbb{Z}^n\rtimes_\sigma G,
\qquad
W'=\mathbb{Z}^m\rtimes_{\sigma'} G',
\]
with $n,m\ge 3$, both satisfying $(\ast)$.
If $W\cong W'$, then
\[
n=m
\qquad\text{and}\qquad
G\cong G'.
\]
\end{coro}

\begin{proof}
This follows immediately from Theorem~\ref{teo:intrinsic_rigidity} by taking $H=K=\mathbb{Z}$.
\end{proof}


\section{Pullback interpretation and functoriality}\label{sec:pullback}

In this section we provide a conceptual interpretation of the permutational wreath pullback and explain its naturality from the viewpoint of group extensions and base change.

\subsection{Groups over the symmetric group and the pullback description}

Let $n\ge 2$ and let $S_n$ denote the symmetric group on $\{1,\dots,n\}$.

\begin{defi}
The category $\mathbf{Grp}_{/S_n}$ of groups over $S_n$ is defined as follows:
\begin{itemize}
\item Objects are pairs $(G,\sigma)$, where $\sigma\colon G \to S_n$ is a group homomorphism.

\item A morphism $f\colon (G_1,\sigma_1) \to (G_2,\sigma_2)$ is a group homomorphism $f\colon G_1 \to G_2$ such that
\[
\sigma_2 \circ f = \sigma_1.
\]
\end{itemize}
\end{defi}

Thus $\mathbf{Grp}_{/S_n}$ can be viewed as the category of groups equipped with a homomorphism to $S_n$, with morphisms preserving these maps.

We now recall that the classical wreath product
\[
H \wr S_n = H^n \rtimes S_n
\]
fits into the canonical short exact sequence
\begin{equation}\label{eq:wreath_extension}
1 \longrightarrow H^n 
\longrightarrow H \wr S_n 
\stackrel{\pi}{\longrightarrow} S_n 
\longrightarrow 1,
\end{equation}
where $\pi(\mathbf{h},\tau)=\tau$.

We make this interpretation precise by showing that the permutational wreath pullback $H \wr_\sigma G$ is obtained from \eqref{eq:wreath_extension} by base change along $\sigma$.

\begin{teo}\label{teo:pullback}
Let $\sigma\colon G \to S_n$ be a homomorphism. Then there is an isomorphism of groups
\[
H \wr_\sigma G
\;\cong\;
G \times_{S_n} (H \wr S_n),
\]
where
\[
G \times_{S_n} (H \wr S_n)
=
\set{(g,(\mathbf{h},\tau)) \in G \times (H \wr S_n)}
{\sigma(g)=\tau}
\]
is the pullback (or fiber product) of $\pi\colon H \wr S_n \to S_n$ along $\sigma\colon G \to S_n$ in $\mathbf{Grp}$, that is, the subgroup of $G \times (H \wr S_n)$ consisting of pairs whose images in $S_n$ coincide.
\end{teo}

\begin{proof}
Define
\[
\Phi \colon H \wr_\sigma G \longrightarrow 
G \times_{S_n} (H \wr S_n)
\]
by
\[
\Phi(\mathbf{h},g)
=
\bigl(g,(\mathbf{h},\sigma(g))\bigr).
\]

Since $\pi(\mathbf{h},\sigma(g))=\sigma(g)$, the pair lies in the pullback. The map $\Phi$ is a homomorphism, since multiplication in both semidirect products is defined using the same permutational action via $\sigma(g)\in S_n$.

It is clearly bijective, with inverse
\[
(g,(\mathbf{h},\sigma(g)))
\longmapsto
(\mathbf{h},g).
\]
Thus $H \wr_\sigma G$ is isomorphic to the fiber product.
\end{proof}

\begin{obs}
Theorem~\ref{teo:pullback} shows that $H \wr_\sigma G$ is not an ad hoc construction, but arises naturally from the canonical permutational extension \eqref{eq:wreath_extension} by base change along $\sigma$.

In particular, many structural properties of $H \wr_\sigma G$ can be understood as inherited from the classical wreath product via this pullback construction. Moreover, the construction depends functorially on the morphism $\sigma\colon G \to S_n$.
\end{obs}

\subsection{Functoriality in $G$ and $H$}

We now turn to the functorial behavior of the construction regarding both the base group $G$ and the coefficient group $H$. We begin with functoriality in the base group $G$.

\begin{prop}\label{prop:functorial_G}
Let $(G_1,\sigma_1)$ and $(G_2,\sigma_2)$ be objects of $\mathbf{Grp}_{/S_n}$ and let $f\colon G_1 \to G_2$ be a morphism in $\mathbf{Grp}_{/S_n}$. Then $f$ induces a natural homomorphism
\[
H \wr_{\sigma_1} G_1 
\longrightarrow 
H \wr_{\sigma_2} G_2
\]
given by
\[
(\mathbf{h},g) \longmapsto (\mathbf{h},f(g)).
\]
\end{prop}

\begin{proof}
The compatibility condition $\sigma_2\circ f=\sigma_1$ ensures that $f(g)$ acts on $H^n$ in the same way as $g$. 
Thus the map preserves multiplication and is a well-defined homomorphism.
\end{proof}

We next consider functoriality with respect to the coefficient group $H$.

\begin{prop}\label{prop:functorial_H}
Let $\varphi\colon H \to K$ be a group homomorphism. Then $\varphi$ induces a natural homomorphism
\[
H \wr_\sigma G 
\longrightarrow 
K \wr_\sigma G
\]
given by
\[
(\mathbf{h},g)
\longmapsto
(\varphi(h_1),\dots,\varphi(h_n),g).
\]
\end{prop}

\begin{proof}
The map $H^n \to K^n$ induced by $\varphi$ is $G$-equivariant, since the action of $G$ permutes coordinates. Hence it extends to a homomorphism of semidirect products.
\end{proof}

Propositions~\ref{prop:functorial_G} and \ref{prop:functorial_H} show that the construction
\[
(H,\sigma\colon G\to S_n) \longmapsto H \wr_\sigma G
\]
is functorial in the coefficient group $H$ and in the pair $(G,\sigma)$. This functorial perspective explains the uniform behavior observed in the structural results of Section~\ref{sec:structure}, and provides a conceptual framework for the braid-type applications developed in the next section.

 
\section{Applications to braid-type groups}\label{sec:braid}

The general framework developed in the previous sections yields, as concrete applications,
a systematic study of framed braid-type groups. 
To the best of our knowledge, even in the classical case, the algebraic structure of framed braid groups—such as their centers, abelianizations, and related invariants—had not been investigated in a unified or conceptual way. 

Moreover, for several families of braid-type groups, including virtual and singular braid groups, no general definition of framed analogues was previously available in the literature. 
The permutational wreath pullback, defined in Section~\ref{sec:wreath}, provides a natural and flexible construction that simultaneously defines these groups and allows for a uniform analysis of their structure. 
We illustrate this by deriving explicit structural results for classical, surface, virtual, and singular framed braid groups. Throughout, we take $H=\mathbb{Z}$ unless otherwise stated. 

We now obtain explicit descriptions of the algebraic structure of framed braid groups in several settings. 
These results follow directly from the general theory developed in this paper, and appear to be new even in classical cases.

\subsection{Classical and surface braid groups}

We begin with the classical case, which serves as the motivating example for the general construction introduced in this paper.
Let $B_n$ denote the classical Artin braid group with its canonical surjective homomorphism
\[
\sigma \colon B_n \to S_n,
\]
sending each standard generator $\sigma_i$ to the transposition $(i,i+1)$.

\begin{figure}[h]
\centering
\includegraphics[width=0.42\textwidth]{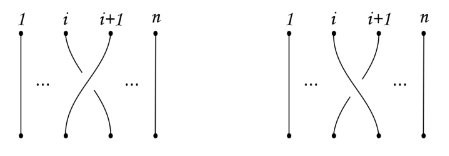}
\caption{The classical Artin generator $\sigma_i$ and its induced permutation in $S_n$.}
\label{fig:classical_generator}
\end{figure}

\begin{obs}
The geometric crossing represented by $\sigma_i$ induces precisely the transposition $(i,i+1)$ on the set of strands. This permutation determines the action of $B_n$ on $\mathbb{Z}^n$ in the semidirect product $\mathbb{Z}^n \rtimes_\sigma B_n$.
\end{obs}

\begin{defi}
The classical framed braid group on $n$ strands is
\[
FB_n := \mathbb{Z}^n \rtimes_\sigma B_n.
\]
\end{defi}

This shows that the classical framed braid group $FB_n$ fits naturally into the framework of permutational wreath pullbacks $\mathbb{Z}\wr_\sigma B_n$. 
We now apply the general structural results of Section~\ref{sec:structure}.

\begin{prop}\label{coro:FBn_structural}
Let $n\ge 3$. Then:
\begin{enumerate}
\item $Z(FP_{n}) = \mathbb{Z}^{n} \times \langle \Delta_{n}^{2} \rangle \cong \mathbb{Z}^{n+1}$, where
\[
\Delta_{n} = (\sigma_{1}\sigma_{2}\cdots\sigma_{n-1})(\sigma_{1}\sigma_{2}\cdots\sigma_{n-2})\cdots(\sigma_{1}\sigma_{2})\sigma_{1} \in B_{n}
\]
is the full twist.

\item $Z(FB_{n})=\mathbb{Z}[\theta]\times \langle\Delta_{n}^{2}\rangle\cong\mathbb{Z}\times \mathbb{Z}$, where $\theta=t_{1}\cdots t_{n}$.

\item For $n = 3$, we get
\[
FB_{3} / Z(FB_{3}) \cong \mathbb{Z}^{2} \rtimes PSL_{2}(\mathbb{Z}).
\]
\item $(FB_n)^{ab}\cong \mathbb{Z}\times \mathbb{Z}$.
\end{enumerate}
\end{prop}

\begin{proof}
Items (1), (2), and (4) follow from Theorems~\ref{teo:center} and~\ref{teo:abelianization}, applied to $H=\mathbb{Z}$ and $G=B_n$, together with the classical descriptions of $Z(B_n)$, $Z(P_n)$ and $B_n^{ab}$. 

For \(n=3\), since \(B_3/Z(B_3)\cong PSL_2(\mathbb Z)\) and \(Z(FB_3)=\Delta(\mathbb Z)\times Z(B_3)\), the quotient description follows from the semidirect product structure.
\end{proof}

\begin{figure}[h]
\centering
\includegraphics[width=0.75\textwidth]{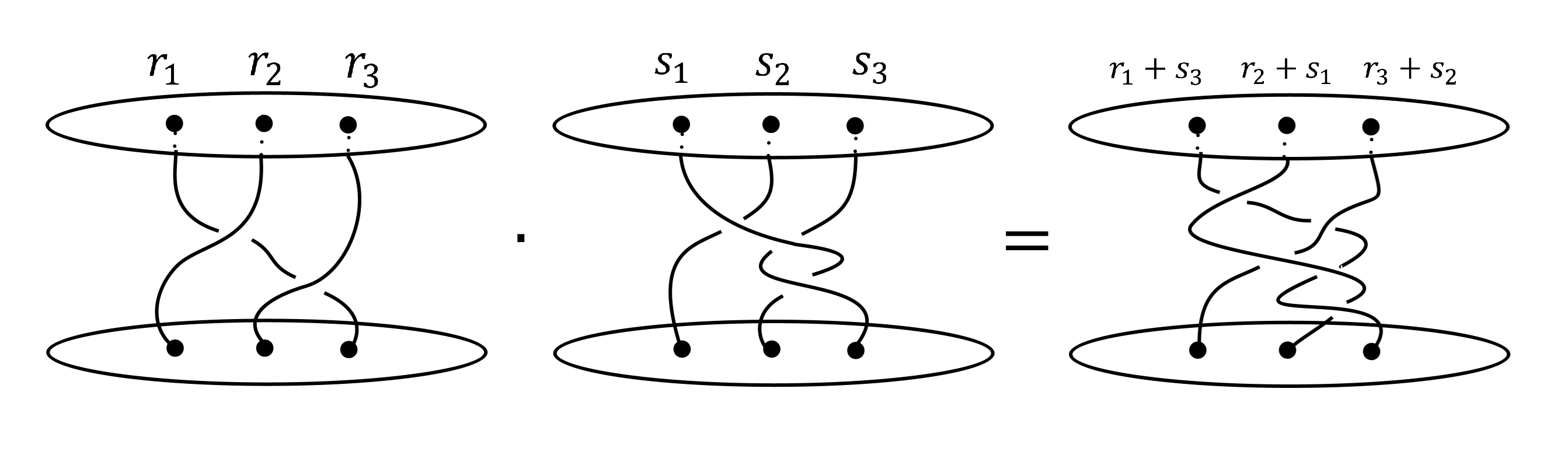}
\caption{Framing interpretation: the $\mathbb{Z}^n$--coordinates record integer twisting data along strands.}
\label{fig:framing_visual}
\end{figure}

\begin{obs}
The additional $\mathbb{Z}^n$ component encodes integer twisting along each strand. 
The permutation action induced by $\sigma$ rearranges these integer coordinates exactly as strands are permuted.
\end{obs}

The same construction extends naturally to braid groups on surfaces. 
Let $M$ be a connected surface and let $B_n(M)$ denote the braid group on the surface $M$. For details about surface braid groups we refer the reader to the references \cite{FvB, FaN, FoN, GPi}. 
There is a canonical surjective homomorphism
\[
\sigma\colon B_n(M)\to S_n
\]
obtained by retaining only the induced permutation of strands.

\begin{obs}
Surface braid generators move strands around handles or punctures. While the geometric motion may be more intricate than in the classical case, the induced permutation on strands still defines a natural surjection to $S_n$, which governs the semidirect product structure.
\end{obs}

\begin{defi}
The framed surface braid group is
\[
FB_n(M):=\mathbb{Z}^n \rtimes_\sigma B_n(M).
\]
\end{defi}

Equivalently, $FB_n(M)=\mathbb{Z}\wr_\sigma B_n(M)$.

\begin{figure}[h]
\centering
\includegraphics[width=0.48\textwidth]{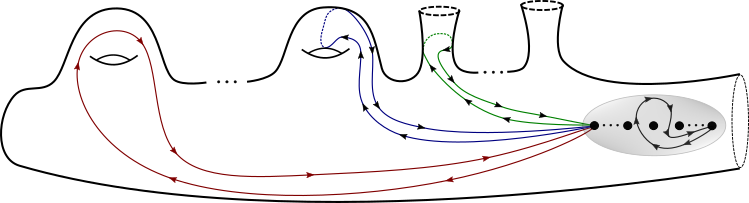}
\caption{Example of a braid in a punctured orientable surface.}
\label{fig:surface_generators}
\end{figure}

Thus, framed surface braid groups arise as permutational wreath pullbacks, and the results of Section~\ref{sec:structure} apply verbatim. 
In particular, whenever structural information on $B_n(M)$ or $P_n(M)$ is available,
our general results immediately yield corresponding information for $FB_n(M)$ and its pure subgroup.

 As in \cite{PR}, a compact surface $M$ will be called \emph{large} if it is different from
\begin{multicols}{3}
\begin{itemize}
	\item the sphere,
	\item the projective plane,
	\item the disk,
	\item the annulus,
	\item the torus,
	\item the Möbius strip, or
	\item the Klein bottle.
\end{itemize}
\end{multicols}
We shall call these seven surfaces \emph{non-large surfaces}. 
Recall that when $M=D^2$ (the disk)  then $B_n(D^2)$ (resp.\ $P_n(D^2)$) is the classical Artin braid group denoted by $B_n$ (resp.\ the classical pure Artin braid group denoted by $P_n$). 
The center of the framed braid group was considered in Proposition~\ref{coro:FBn_structural}. Now, we show a result about the center of framed surface braid groups $FB_n(M)$, for any large surface $M$.

\begin{prop}
Let $M$ be a compact large surface and let $\theta=t_1\cdots t_n\in\mathbb Z^n$.
If $n\ge 2$, then
\[
Z(FP_n(M))\cong \mathbb Z^n,
\qquad
Z(FB_n(M))\cong \mathbb Z.
\]
More precisely,
\[
Z(FP_n(M))=\mathbb Z^n,
\qquad
Z(FB_n(M))=\langle\theta\rangle.
\]
\end{prop}

\begin{proof}
This is an immediate consequence of Theorem~\ref{teo:center} together with the description of the centers of braid groups of large surfaces, see \cite[Proposition~1.6]{PR}.
\end{proof}

\begin{obs}

It is also possible to obtain explicit descriptions of the center of $FB_n(M)$ when $M$ is a non-large surface by combining Theorem~\ref{teo:center} with the known descriptions of the centers of the corresponding surface braid groups. 
Since these depend on the specific topology of the surface and require a case-by-case analysis, we omit them here. 
For instance, if $M=\mathbb S^2$ and $n\ge 3$, it was proved in \cite{GvB} that
\[
Z\bigl(B_n(\mathbb S^2)\bigr)\cong \mathbb Z_2,
\]
generated by $\Delta_n^2$. Hence, by Theorem~\ref{teo:center}, we obtain
\[
Z\bigl(FB_n(\mathbb S^2)\bigr)
=
\langle \theta\rangle \times \langle \Delta_n^2\rangle
\cong
\mathbb Z \times \mathbb Z_2,
\]
where $\theta=t_1\cdots t_n$.
\end{obs}

\subsubsection{Relation with surface framed braids}

The permutational wreath pullback construction is closely related to the framed braid groups over surfaces introduced by Bellingeri and Gervais~\cite{BG}, although the two approaches arise from different perspectives.

Let $M$ be a compact surface and consider the classical short exact sequence
\[
1 \longrightarrow P_n(M) \longrightarrow B_n(M) \stackrel{\sigma}{\longrightarrow} S_n \longrightarrow 1.
\]
In~\cite{BG}, framed braid groups over $M$ are defined using geometric methods, involving configuration spaces, tangent bundle considerations, and, in some formulations, mapping class group techniques. As a consequence, the resulting groups may depend on geometric features of the surface, such as orientability or the triviality of the tangent bundle, and the construction is not uniform across all surfaces.

In many cases, however, their construction yields a semidirect product of the form
\[
\mathbb{Z}^n \rtimes B_n(M),
\]
where the action is induced by permutation of strands. In these situations, the resulting group is abstractly isomorphic to the permutational wreath pullback
\[
\mathbb{Z} \wr_\sigma B_n(M) = \mathbb{Z}^n \rtimes_\sigma B_n(M).
\]

The key difference is conceptual: the approach of~\cite{BG} is geometric and surface-dependent, whereas the present construction is purely algebraic and depends only on the permutation representation $\sigma\colon G\to S_n$. Thus, the permutational wreath pullback provides a uniform framework that applies to any group equipped with such a representation, independently of any underlying geometric structure.

\subsection{Virtual and singular braid groups}\label{subsec:vbn}

We next consider virtual braid groups, which provide a broader combinatorial setting. Let $VB_n$ denote the virtual braid group introduced by Kauffman \cite{Kau}. 
The virtual braid group admits two natural surjective homomorphisms onto $S_n$ (see \cite[Section~2]{BB}, \cite[Section~2]{BP}):

\begin{itemize}
    \item $\pi_K\colon VB_n \to S_n$ defined by $\pi_K(\sigma_i)=1$ and $\pi_K(v_i)=s_i$, where $s_i=(i,i+1)$;
    \item $\pi_P\colon VB_n \to S_n$ defined by $\pi_P(\sigma_i)=\pi_P(v_i)=s_i$.
\end{itemize}

Their kernels are denoted by
\[
KB_n = \ker{\pi_K},\qquad VP_n = \ker{\pi_P}.
\]
The map $\iota\colon S_n\to VB_n$ given by $\iota(s_i)=v_i$ is a common section for both epimorphisms, yielding the semidirect product decompositions
\[
VB_n = KB_n \rtimes S_n = VP_n \rtimes S_n.
\]

\begin{figure}[h]
\centering
\begin{minipage}{0.45\textwidth}
    \centering
    \includegraphics[width=0.6\textwidth]{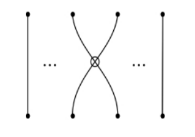}
    \caption{Virtual braid generator.}
    \label{fig:virtual_generator}
\end{minipage}
\hfill
\begin{minipage}{0.45\textwidth}
    \centering
    \includegraphics[width=0.5\textwidth]{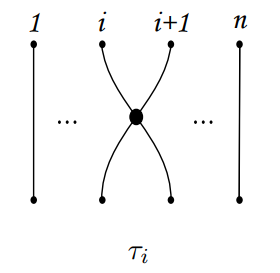}
    \caption{Singular braid generator.}
    \label{fig:singular_generator}
\end{minipage}
\end{figure}

\begin{obs}
Although virtual crossings have a distinct diagrammatic nature, their projection under $\pi_P$ still yields the same symmetric group action, hence the same permutational mechanism governs framings.
\end{obs}

\begin{defi}
The framed virtual braid group is defined as 
\[
FVB_n:=\mathbb{Z}^n \rtimes_\sigma VB_n.
\]
\end{defi}

As in the classical case, the canonical projection to $S_n$ induces the permutational action on $\mathbb{Z}^n$. 
Therefore, all structural and rigidity results obtained in Section~\ref{sec:structure} apply to $FVB_n$.

Thus $FVB_n=\mathbb{Z}\wr_\sigma VB_n$, and
\[
(FVB_n)^{ab}\cong \mathbb{Z}\times VB_n^{ab}.
\]

Finally, we consider singular braid groups, which introduce additional generators while preserving the underlying permutation structure. 
Let $SG_n$ denote the singular braid group \cite{FKR}. There is a natural surjective homomorphism
\[
\sigma\colon SG_n\to S_n
\]
induced by the underlying permutation of strands.

\begin{obs}
Singular crossings introduce additional generators but do not alter the underlying permutation representation, so the framing construction remains identical.
\end{obs}

\begin{defi}
The framed singular braid group is
\[
FSG_n:=\mathbb{Z}^n \rtimes_\sigma SG_n.
\]
\end{defi}
Once again, the permutational nature of the action ensures that the framed singular braid group fits into our general framework.

\begin{prop}
Let $n\geq 3$. Then
\begin{enumerate}

\item $(FVB_n)^{ab}\cong \mathbb Z\times VB_n^{ab}$. 

\item $(FSG_n)^{ab}\cong \mathbb Z\times SG_n^{ab}$.

    \item $Z(FVP_n) = \mathbb{Z}^n = \ang{\{t_1, t_2, \ldots, t_n \}}$,
    \item $Z(FVB_n) = \mathbb{Z} = \ang{t_1\cdot t_2 \cdots t_n }$.

    \item $Z(FSP_{n}) = \mathbb{Z}^{n} \times \langle \Delta_{n}^{2} \rangle \cong \mathbb{Z}^{n+1}$, where $\Delta_{n} = (\sigma_{1}\sigma_{2}\cdots\sigma_{n-1})(\sigma_{1}\sigma_{2}\cdots\sigma_{n-2})\cdots(\sigma_{1}\sigma_{2})\sigma_{1} \in B_{n}$ is the full twist.

\item $Z(FSG_{n})=\mathbb{Z}[\theta]\times \langle\Delta_{n}^{2}\rangle\cong\mathbb{Z}\times \mathbb{Z}$, where $\theta=t_{1}\cdots t_{n}$.
    \end{enumerate}
\end{prop}

\begin{proof}
It follows from Theorem~\ref{teo:abelianization}, Theorem~\ref{teo:center} and \cite[Theorems~6 and~11]{DN} for virtual braid groups and from \cite{V2} for singular braid groups. 
\end{proof}

We also may obtain similar results for the virtual twin group $VT_n$ (see \cite[Section~2]{NNS}, \cite[Section~2]{NNS2}), that is a planar analogue of the virtual braid group. 
It admits two natural surjective homomorphisms onto $S_n$:
\begin{itemize}
    \item $\sigma\colon VT_n \to S_n$ defined by $\sigma(s_i)=\sigma(\rho_i)=s_i$;
    \item $\theta\colon VT_n \to S_n$ defined by $\theta(s_i)=1$ and $\theta(\rho_i)=s_i$.
\end{itemize}

Their kernels are denoted by
\[
PVT_n = \ker{\sigma},\qquad KT_n = \ker{\theta}.
\]
The map $\iota\colon S_n\to VT_n$ given by $\iota(s_i)=\rho_i$ is a common section, giving the semidirect product decompositions
\[
VT_n = PVT_n \rtimes S_n = KT_n \rtimes S_n.
\]

\subsection{Braid-type examples satisfying condition $(\ast\ast)$}\label{subsec:ast_ast}

We now verify condition $(\ast\ast)$ for the kernels introduced in Section~\ref{subsec:vbn}, for virtual braid and virtual twin groups. 
The relevant epimorphisms are $\pi_K\colon VB_n\to S_n$, $\sigma\colon VT_n\to S_n$, and $\theta\colon VT_n\to S_n$, with kernels $KB_n$, $PVT_n$, and $KT_n$, respectively. 

The structural properties of these kernels are well understood:
\begin{itemize}
    \item $KB_n$ is an Artin group \cite[Proposition~17]{BB}, \cite[Proposition~3.1]{BP};
    \item $PVT_n$ is a right-angled Artin group (RAAG) \cite[Theorem~3.3]{NNS};
    \item $KT_n$ is a right-angled Coxeter group \cite[Theorem~3.3]{NNS2}.
\end{itemize}
In each case, the action of $S_n$ on the kernel is given by permutation of the indices of the generators, and is faithful.

We will verify condition $(\ast\ast)$ for the kernels $KB_n$, $PVT_n$ and $KT_n$ using different structural arguments. 
The verification for $KB_n$ relies on the amalgamated product structure of Artin groups; for $PVT_n$, we use the RAAG structure and Lemma~\ref{lem:raag_abelian_normal}; for $KT_n$, we use the classification of affine Coxeter groups and results on amenable normal subgroups.

\subsubsection{Technical lemmas on RAAGs}

We begin by establishing some technical lemmas about right-angled Artin groups, which will be used in the verification of condition $(\ast\ast)$ for virtual twin groups.

\begin{lema}\label{lem:raag_two_generator}
Let $G$ be a right-angled Artin group. If $a,x\in G$ satisfy $[a,x]\neq 1$, then the subgroup $\langle a,x\rangle$ is a non-abelian free group of rank $2$.
\end{lema}

\begin{proof}
Baudisch proved that every subgroup generated by two elements of a semifree group is either free or free abelian \cite{Bau}. Since semifree groups coincide with right-angled Artin groups (see \cite[p.~2]{Charney}), it follows that $\langle a,x\rangle$ is either free abelian or free. As $a$ and $x$ do not commute, $\langle a,x\rangle$ is not abelian, hence it must be a free group. Since it is generated by two elements, it is a free group of rank $2$.
\end{proof}

\begin{lema}\label{lem:free_group_centralizer}
Let $F$ be a non-abelian free group and let $1\neq a\in F$. Then:
\begin{enumerate}
\item the centralizer $C_F(a)$ is an infinite cyclic group;
\item the normalizer of the cyclic subgroup $\langle a\rangle$ coincides with $C_F(a)$.
\end{enumerate}
\end{lema}

\begin{proof}
These are standard facts about free groups; see, for instance, \cite[Chapter~I, Section~4]{MKS}.
\end{proof}

\begin{lema}\label{lem:raag_abelian_normal}
Let $G$ be a right-angled Artin group with trivial center. Then $G$ contains no non-trivial abelian normal subgroup.
\end{lema}

\begin{proof}
Let $A\lhd G$ be an abelian normal subgroup. Suppose that $A\neq \{1\}$, and choose $1\neq a\in A$.

Since $Z(G)=\{1\}$, there exists $x\in G$ such that $[a,x]\neq 1$. By Lemma~\ref{lem:raag_two_generator}, the subgroup
\[
F:=\langle a,x\rangle
\]
is a non-abelian free group.

Since $A$ is normal in $G$, we have
\[
xax^{-1}\in A.
\]
As $A$ is abelian, the elements $a$ and $xax^{-1}$ commute. Hence
\[
xax^{-1}\in C_F(a).
\]
Therefore
\[
C_F(xax^{-1})=C_F(a).
\]
On the other hand,
\[
C_F(xax^{-1})=x\,C_F(a)\,x^{-1}.
\]
Thus
\[
x\,C_F(a)\,x^{-1}=C_F(a),
\]
so $x$ normalizes $C_F(a)$.

By Lemma~\ref{lem:free_group_centralizer}, the group $C_F(a)$ is infinite cyclic, and its normalizer in $F$ coincides with $C_F(a)$. Hence
\[
x\in C_F(a),
\]
which implies $[a,x]=1$, a contradiction.

Therefore $A=\{1\}$.
\end{proof}

The previous lemmas will be used to control abelian normal subgroups in right-angled Artin groups with trivial center. This will allow us to treat the kernel $PVT_n$ in a uniform way within the framework of condition $(\ast\ast)$. 

\subsubsection{Verification of condition $(\ast\ast)$ and the $R_\infty$-property}

We now verify condition $(\ast\ast)$ for several families of braid-type groups, using different structural arguments in each case.

\begin{teo}\label{teo:double_star}
Let $n\ge 3$. Then condition $(\ast\ast)$ holds in each of the following cases:
\begin{enumerate}
\item $(VB_n,\pi_K)$;
\item $(VT_n,\sigma)$;
\item $(VT_n,\theta)$.
\end{enumerate}
Equivalently, the following statements hold:
\begin{enumerate}
\item $KB_n$ contains no non-trivial abelian normal subgroup of $VB_n$;
\item $PVT_n$ contains no non-trivial abelian normal subgroup of $VT_n$;
\item $KT_n$ contains no non-trivial abelian normal subgroup of $VT_n$.
\end{enumerate}
\end{teo}

\begin{proof}
Let $n\ge 3$. We treat each case separately. 
    \begin{enumerate}
        \item Let $N\triangleleft VB_n$ be an abelian subgroup with $N\subseteq KB_n$. 
Since $VB_n = KB_n \rtimes S_n$, normality of $N$ in $VB_n$ implies that $N$ is invariant under conjugation by $S_n$, and in particular, $N$ is invariant under each reflection $r_\beta\in S_n$ corresponding to a positive root $\beta$.

By \cite[Proposition~17]{BB} and \cite[Proposition~3.1]{BP}, the group $KB_n$ is an Artin group generated by
\[
\{\delta_{i,j} \mid 1\le i\neq j\le n\},
\]
and $S_n$ acts by permutation of indices:
\[
w\cdot \delta_{i,j} = \delta_{w(i),w(j)}.
\]

Let $\Gamma=A_{n-1}$ and let $\Phi^+[\Gamma]$ be the set of positive roots.
For each $\beta\in\Phi^+[\Gamma]$, define
\[
\mathcal{X}_\beta^+ = \Phi[\Gamma]\setminus\{-\beta\},\quad
\mathcal{X}_\beta^- = \Phi[\Gamma]\setminus\{\beta\},\quad
\mathcal{Y}_\beta = \Phi[\Gamma]\setminus\{\beta,-\beta\}.
\]
By \cite[Lemma~3.3]{BP}, we have
\[
KB_n = A[\widehat{\Gamma}_{\mathcal{X}_\beta^+}] *_{A[\widehat{\Gamma}_{\mathcal{Y}_\beta}]} A[\widehat{\Gamma}_{\mathcal{X}_\beta^-}],
\]
where the two factors are interchanged by the reflection $r_\beta$.

Since $N$ is abelian and invariant under $r_\beta$, the standard argument for Artin groups associated to Coxeter systems (see \cite[Lemma~3.9]{BPT}, which handles abelian subgroups invariant under an involution swapping the factors, or, using the fact that $N$ is abelian, \cite[Lemma~3.6]{BP}) shows that, for every $\beta\in\Phi^+[\Gamma]$, every element of $N$ lies in the intersection of the two factors, namely in $A[\widehat{\Gamma}_{\mathcal{Y}_\beta}]$. 
Hence
\[
N \subseteq \bigcap_{\beta\in\Phi^+[\Gamma]} A[\widehat{\Gamma}_{\mathcal{Y}_\beta}].
\]

The intersection of all $\mathcal{Y}_\beta$ is empty. 
Indeed, for any $\gamma\in\Phi[\Gamma]$, either $\gamma$ or $-\gamma$ is a positive root, and thus $\gamma$ is excluded from the corresponding $\mathcal{Y}_\beta$.
By \cite[Theorem~3.2]{BP},
\[
\bigcap_{\beta\in\Phi^+[\Gamma]} A[\widehat{\Gamma}_{\mathcal{Y}_\beta}]
= A[\widehat{\Gamma}_{\emptyset}] = \{1\}.
\]

Thus $N=\{1\}$.

        \item Let $N\triangleleft VT_n$ be an abelian subgroup such that $N\subseteq PVT_n$.

By \cite[Section~2]{NNS}, one has
\[
VT_n=PVT_n\rtimes S_n.
\]
Thus $N$ is normal in $PVT_n$. 
By \cite[Theorem~3.3]{NNS}, the group $PVT_n$ is a right-angled Artin group, and by \cite[Corollary~4.2]{NNS}, its center is trivial for $n\ge 3$. Therefore, Lemma~\ref{lem:raag_abelian_normal} applies to $PVT_n$ and yields
\[
N=\{1\}.
\]

        \item Let $N\triangleleft VT_n$ be an abelian subgroup such that $N\subseteq KT_n$.
Since $KT_n\le VT_n$ and $N$ is normal in $VT_n$, it follows that $N$ is also normal in $KT_n$.

By \cite[Theorem~3.3]{NNS2} and \cite[Corollary~3.4]{NNS2}, the group $KT_n$ is an irreducible right-angled Coxeter group of rank $n(n-1)$ and with trivial center. 
In particular, for $n\ge 3$, the rank of $KT_n$ is at least $6$. 
The classification of irreducible affine Coxeter groups shows that the only irreducible affine right-angled Coxeter group is the infinite dihedral group $D_\infty$ (see \cite[Section~2]{Dani}). Hence, it follows that $KT_n$ is not affine.

Therefore $KT_n$ is an infinite irreducible non-affine Coxeter group. 
It is a theorem of de Cornulier \cite[Corollary~1.2]{dC} and L\'ecureux \cite[Theorem~1.1]{Lec} that every infinite irreducible non-affine Coxeter group has no non-trivial amenable normal subgroup. 
Since every abelian group is amenable, it follows that $N=\{1\}$.
    \end{enumerate}
\end{proof}

Theorem~\ref{teo:double_star} provides a unified verification of condition $(\ast\ast)$ for several braid-type groups. Combined with the equivalence between conditions $(\ast)$ and $(\ast\ast)$ for $n\ge 5$, this yields the following corollaries.

\begin{coro}\label{coro:star_and_double_star}
Let $n\ge 5$. Then conditions $(\ast)$ and $(\ast\ast)$ hold in each of the following cases:
\begin{enumerate}
\item $(VB_n,\pi_K)$;
\item $(VT_n,\sigma)$;
\item $(VT_n,\theta)$.
\end{enumerate}
\end{coro}

\begin{proof}
By Theorem~\ref{teo:double_star}, condition $(\ast\ast)$ holds in each of the three cases.

Since $n\ge 5$, Theorem~\ref{teo:conditionA_equiv} applies, and shows that conditions $(\ast)$ and $(\ast\ast)$ are equivalent. Hence condition $(\ast)$ also holds in each case.
\end{proof}

As a further consequence, we obtain new families of groups with the $R_\infty$-property.

\begin{coro}\label{coro:R_infty_examples}
Let $H$ be a finitely generated abelian group and let $n\ge 5$. Then each of the following permutational wreath pullbacks has the $R_\infty$-property:
\begin{enumerate}
\item $H^n\rtimes_{\pi_K} VB_n$;
\item $H^n\rtimes_{\sigma} VT_n$;
\item $H^n\rtimes_{\theta} VT_n$.
\end{enumerate}
\end{coro}

\begin{proof}
By Corollary~\ref{coro:star_and_double_star}, condition $(\ast)$ holds in each of the three cases.

Moreover, both $VB_n$ and $VT_n$ have the $R_\infty$-property by \cite{DGO2}. Therefore, in each case, Theorem~\ref{teo:R_infty_general} implies that the corresponding permutational wreath pullback has the $R_\infty$-property.
\end{proof}

\begin{obs}
We conclude this subsection by discussing a natural remaining case.
It is natural to ask whether condition $(\ast\ast)$ also holds for the kernel
\[
VP_n = \ker{\pi_P}
\]
of the canonical epimorphism $\pi_P\colon VB_n \to S_n$.

Although $VP_n$ shares several structural features with the groups considered above, such as a presentation indexed by pairs of strands and a natural action of $S_n$, the methods used in Theorem~\ref{teo:double_star} do not readily extend to this case. In particular, a suitable analogue of the parabolic decomposition used for $KB_n$, or of the RAAG and Coxeter arguments used for $PVT_n$ and $KT_n$, is not currently available in a form that would allow us to conclude.

We therefore leave the verification of condition $(\ast\ast)$ for $VP_n$ as an interesting open problem.
\end{obs}

\subsection{Sections of forgetful maps}

In this subsection we study the behavior of forgetful maps at the level of the subgroup lying over the permutational kernel. 
This is the natural setting for braid-type applications, where such maps arise on pure braid groups through Fadell--Neuwirth type constructions.

Let
\[
\sigma_{n+m}\colon G_{n+m}\to S_{n+m},
\qquad
\sigma_n\colon G_n\to S_n
\]
be surjective homomorphisms, and set
\[
P_{\sigma_{n+m}}:=\ker{\sigma_{n+m}},
\qquad
P_{\sigma_n}:=\ker{\sigma_n}.
\]
Let
\[
f\colon P_{\sigma_{n+m}}\longrightarrow P_{\sigma_n}
\]
be a surjective homomorphism, and let
\[
\psi_m\colon H^{n+m}\longrightarrow H^n
\]
be the projection onto the first \(n\) factors. We assume that $f$ is induced by a forgetful-type construction on the underlying braid-type groups.

For \(k=n,n+m\), consider the associated subgroups
\[
PW_k=\Pi_k^{-1}(P_{\sigma_k})\le H^k\rtimes_{\sigma_k} G_k,
\]
where \(\Pi_k\colon H^k\rtimes_{\sigma_k} G_k\to G_k\) denotes the canonical projection.
By Proposition~\ref{prop:pure_splits}, we have natural identifications
\[
PW_{n+m}\cong H^{n+m}\times P_{\sigma_{n+m}},
\qquad
PW_n\cong H^n\times P_{\sigma_n}.
\]

Via these identifications, the homomorphism $f$ naturally induces
\[
(\psi_m,f)\colon PW_{n+m}\longrightarrow PW_n,
\qquad
(\mathbf{h},g)\longmapsto (\psi_m(\mathbf{h}),f(g)).
\]

\begin{prop}\label{prop:sections_forgetful}
The homomorphism
\[
(\psi_m,f)\colon PW_{n+m}\longrightarrow PW_n
\]
admits a section if and only if the homomorphism 
\[
f\colon P_{\sigma_{n+m}}\longrightarrow P_{\sigma_n}
\]
admits a section.
\end{prop}

\begin{proof}
Assume first that \(f\) admits a section
\[
s\colon P_{\sigma_n}\longrightarrow P_{\sigma_{n+m}}.
\]
Define
\[
\widetilde{s}\colon PW_n\longrightarrow PW_{n+m}
\]
by
\[
\widetilde{s}(h_1,\dots,h_n,p)
=
(h_1,\dots,h_n,1,\dots,1,s(p)),
\]
where the last \(m\) entries in the \(H\)-component are equal to the identity element of \(H\).
Then
\[
(\psi_m,f)\circ \widetilde{s}=\mathrm{id}_{PW_n},
\]
so \((\psi_m,f)\) admits a section.

Conversely, suppose that
\[
\widetilde{s}\colon PW_n\longrightarrow PW_{n+m}
\]
is a section of \((\psi_m,f)\).
Let
\[
\iota\colon P_{\sigma_n}\longrightarrow PW_n,
\qquad
\iota(p)=(1,p),
\]
and let
\[
\mathrm{pr}_2\colon PW_{n+m}\cong H^{n+m}\times P_{\sigma_{n+m}}
\longrightarrow P_{\sigma_{n+m}}
\]
denote the projection onto the second factor.
Define
\[
s:=\mathrm{pr}_2\circ \widetilde{s}\circ \iota
\colon
P_{\sigma_n}\longrightarrow P_{\sigma_{n+m}}.
\]
Then \(s\) is a homomorphism, and
\[
f\circ s
=
\mathrm{pr}_2\circ (\psi_m,f)\circ \widetilde{s}\circ \iota
=
\mathrm{pr}_2\circ \iota
=
\mathrm{id}_{P_{\sigma_n}}.
\]
Hence \(f\) admits a section.
\end{proof}

The previous proposition shows that, at the level of the subgroup lying over the permutational kernel, the framing coordinates do not introduce new obstructions to the existence of sections.

\begin{coro}\label{coro:surface_sections}
Let \(M\) be a connected surface for which the Fadell--Neuwirth forgetful homomorphism
\[
p_\ast\colon P_{n+m}(M)\longrightarrow P_n(M)
\]
is defined. Let
\[
FB_{n+m}(M)=\mathbb{Z}^{n+m}\rtimes B_{n+m}(M),
\qquad
FB_n(M)=\mathbb{Z}^n\rtimes B_n(M),
\]
and let
\[
\Pi_{n+m}\colon FB_{n+m}(M)\to B_{n+m}(M),
\qquad
\Pi_n\colon FB_n(M)\to B_n(M)
\]
be the canonical projections. Set
\[
FP_{n+m}(M):=\Pi_{n+m}^{-1}\bigl(P_{n+m}(M)\bigr),
\qquad
FP_n(M):=\Pi_n^{-1}\bigl(P_n(M)\bigr).
\]
Then the induced homomorphism
\[
(\psi_m,p_\ast)\colon FP_{n+m}(M)\longrightarrow FP_n(M)
\]
admits a section if and only if
\[
p_\ast\colon P_{n+m}(M)\longrightarrow P_n(M)
\]
admits a section.
\end{coro}

\begin{proof}
Apply Proposition~\ref{prop:sections_forgetful} with
\[
G_{n+m}=B_{n+m}(M),\qquad G_n=B_n(M),\qquad H=\mathbb{Z},
\]
and with \(\sigma_{n+m}\) and \(\sigma_n\) the canonical permutation homomorphisms. In this case
\[
P_{\sigma_{n+m}}=P_{n+m}(M),
\qquad
P_{\sigma_n}=P_n(M),
\]
and Proposition~\ref{prop:pure_splits} gives
\[
FP_{n+m}(M)\cong \mathbb{Z}^{n+m}\times P_{n+m}(M),
\qquad
FP_n(M)\cong \mathbb{Z}^n\times P_n(M).
\]
The claim follows immediately.
\end{proof}

\begin{obs}
Corollary~\ref{coro:surface_sections} shows that the splitting problem for framed surface braid groups reduces entirely to the corresponding problem for pure surface braid groups.

The latter has been completely solved by Gonçalves and Guaschi~\cite[Theorem~2]{GG2010}, who determine precisely when the Fadell--Neuwirth homomorphism
\[
p_\ast\colon P_{n+m}(M)\longrightarrow P_n(M)
\]
admits a section, for arbitrary compact surfaces $M$.

Therefore, the existence of sections for the framed forgetful homomorphism
\[
(\psi_m,p_\ast)\colon FP_{n+m}(M)\longrightarrow FP_n(M)
\]
is completely governed by this classification. In particular, no new obstructions arise in the framed setting.

This shows that the framed construction preserves the geometric complexity of the splitting problem without introducing additional algebraic obstructions.
\end{obs}

\subsection{Unified structural interpretation}

All the framed braid-type groups above arise from the canonical extension
\[
1\to \mathbb{Z}^n \to \mathbb{Z}\wr S_n \to S_n \to 1
\]
by base change along the respective homomorphisms $\sigma\colon G\to S_n$.

\begin{obs}
The pullback perspective of Theorem~\ref{teo:pullback} shows that framings are not ad hoc additions but rather base changes of the canonical permutational extension. 
This explains the uniform algebraic behavior observed across classical, surface, virtual, and singular braid theories.
\end{obs}

The examples of this section illustrate that framed braid-type groups are not isolated constructions, but rather arise uniformly as permutational wreath pullbacks. This provides a conceptual explanation for their shared structural properties.


\bibliographystyle{alpha}

\end{document}